\documentclass[12pt, reqno]{amsart}
\usepackage{amssymb}
\usepackage{amsfonts}
\usepackage{amsmath}
\usepackage{color}
\usepackage{hyperref}

\newtheorem{theorem}{Theorem}[section]

\newtheorem{corollary}[theorem]{Corollary}

\def\N{\mathbb{N}}
\def\Z{\mathbb{Z}}

\def\P{\mathcal{P}}
\def\B{\mathcal{B}}
\def\D{\mathcal{D}}
\def\K{\mathcal{K}}
\def\T{\mathcal{T}}

\def\Aut{\mathop{\mathrm{Aut}}}

\begin{document}

\title{Some new designs with prescribed automorphism groups}

\author{Vedran Kr\v{c}adinac}

\email{vedran.krcadinac@math.hr}

\address{Department of Mathematics, Faculty of Science, University of Zagreb,
Bijeni\v{c}ka~30, HR-10000 Zagreb, Croatia}

\thanks{The author is partially supported by the Croatian Science Foundation under
project 1637.}



\date{August 27, 2017}

\begin{abstract}
We establish the existence of simple designs with parameters
$2$-$(55,10,4)$, $3$-$(20,5,4)$, $3$-$(21,7,30)$, $4$-$(15,5,2)$,
$4$-$(16,8,45)$, $5$-$(16,7,10)$, and $5$-$(17,8,40)$, which have
previously been unknown. For the corresponding $t$, $v$, and $k$, we
study the set of all $\lambda$ for which simple $t$-$(v,k,\lambda)$
designs exist.
\end{abstract}

\maketitle

\section{Introduction}

A $t$-$(v,k,\lambda)$ \emph{design} $\D$ is a $v$-element set $\P$
of \emph{points} together with a collection $\B$ of $k$-element
subsets called \emph{blocks}, such that every $t$-element subset of
points is contained in exactly $\lambda$ blocks. The design is
\emph{simple} if $\B$ is a set, i.e.\ contains no repeated blocks.
We refer to the Handbook of Combinatorial Designs~\cite{CD07} for
definitions and results about designs. In this paper we are
concerned with the existence problem for simple designs with small
parameters.

An \emph{automorphism} of $\D$ is a permutation of points leaving
$\B$ invariant. The set of all automorphisms forms a group under
composition, the \emph{full automorphism group} $\Aut(\D)$. By
prescribing suitable subgroups $G\le \Aut(\D)$, we are able to
construct simple designs with parameters $2$-$(55,10,4)$,
$3$-$(20,5,4)$, $3$-$(21,7,30)$, $4$-$(15,5,2)$, $4$-$(16,8,45)$,
$5$-$(16,7,10)$, and $5$-$(17,8,40)$. These designs are designated
as unknown in \cite[Table~1.35]{MR07} and \cite[Table~4.46]{KL07}.
Since the Handbook was published, new existence results about
designs with small parameters have appeared in \cite{AH10},
\cite{AHTW10}, \cite{OP08}, \cite{SST14}, and \cite{TT17}. To the
best of our knowledge, existence of the constructed designs has
previously been open.

The layout of our paper is as follows. In the next section we
describe the construction method, the used computational tools, and
some other preliminary matters. Sections \ref{secresfirst} to
\ref{secreslast} are dedicated to designs with particular $t$, $v$,
and $k$. We present our new constructions and try to determine the
set of all $\lambda$ such that simple $t$-$(v,k,\lambda)$ designs
exist. This is accomplished for $4$-$(15,5,\lambda)$ designs in
Section~\ref{sec4-15-5}, and in the other sections up to three open
cases remain. The prescribed groups~$G$ and some computational
details are laid out in the proofs of the theorems. Designs are
presented by listing base blocks; $\B$ is the union of the
corresponding $G$-orbits. The groups and the base blocks for the new
designs are also available on our web page:
\begin{center}{\small
\url{https://web.math.pmf.unizg.hr/~krcko/results/newdesigns.html}
}\end{center}

\section{Preliminaries}

Let $G$ be a group of permutations of $\P=\{1,\ldots,v\}$, and let
$\T_1,\ldots,\T_m$ and $\K_1,\ldots,\K_n$ be the orbits of
$t$-element subsets and $k$-element subsets of $\P$, respectively.
Let $a_{ij}$ be the number of subsets $K\in\K_j$ containing a given
$T\in \T_i$. This number does not depend on the choice of $T$. The
matrix $A=[a_{ij}]$ is the \emph{Kramer-Mesner matrix}. It is well
known that simple $t$-$(v,k,\lambda)$ designs with~$G$ as an
automorphism group exist if and only if the system of linear
equations $A\cdot x=\lambda j$ has $0$-$1$ solutions
$x\in\{0,1\}^n$. Here, $j=(1,\ldots,1)^\tau$ is the all-one vector
of length $m$. This method of construction was pioneered by
E.~S.~Kramer and D.~M.~Mesner~\cite{KM76} and has since been used to
find many new designs with prescribed automorphism groups (see,
e.g.,\ \cite{AHTW10}, \cite{BLW99}, \cite{KNP11}, \cite{KV16},
\cite{KR90}, and \cite{AW98}).

We use GAP~\cite{GAP4} to compute the orbits $\T_i$, $\K_j$ and the
matrix~$A$. For most of our results this is a small and easy
computation. Exceptions are Theorems~\ref{tm1} and~\ref{tm2}, where
we use an algorithm described in~\cite{KV16} to produce short
orbits, and a program written in C for long orbits of the group
$G_3\cong M_{11}$. The second step of the computation is finding
solutions of the Kramer-Mesner system $A\cdot x=\lambda j$. Solving
systems of linear equations over $\{0,1\}$ is a known NP complete
problem. Our prescribed automorphism groups $G$ lead to systems of
small to moderate size, that can be solved fairly quickly. We use
the backtracking solver developed in~\cite{KNP11} for $2$-designs
and A.~Wasserman's solver~\cite{AW98} based on the LLL algorithm for
designs with $t\ge 3$. Running times varied from a few seconds to
several days of CPU time. Finally, to decide whether the constructed
designs are isomorphic and to compute their full automorphism
groups, we use \emph{nauty}/\emph{Traces} by B.~D.~McKay and
A.~Piperno~\cite{MP14}.

A $t$-$(v,k,\lambda)$ design $\D$ is also a $s$-$(v,k,\lambda_s)$
design, for $\lambda_s=\lambda {v-s \choose t-s} / {k-s\choose
t-s}$, $0\le s\le t$. The number of blocks of $\D$ is
$b=|\B|=\lambda_0$. The parameters $t$-$(v,k,\lambda)$ are called
\emph{admissible} if all the $\lambda_s$ are integers, and
\emph{realizable} if simple designs with these parameters exists.
Given $t$, $v$, and $k$, there is a least integer $\lambda_{\min}$
such that $t$-$(v,k,\lambda_{\min})$ are admissible. Any $\lambda$
for which $t$-$(v,k,\lambda)$ are admissible is of the form
$\lambda=m \lambda_{\min}$, for $m\in\N$. The largest $\lambda$ for
which a simple $t$-$(v,k,\lambda)$ design exists is
$\lambda_{\max}={v-t \choose k-t}$. The corresponding
\emph{complete} $t$-$(v,k,\lambda_{\max})$ designs contains all
$k$-element subsets as blocks: $\B={\P \choose k}$.

For the parameters of our newly constructed designs, we try to
determine all $\lambda$ between $\lambda_{\min}$ and
$\lambda_{\max}$ such that $t$-$(v,k,\lambda)$ are realizable. The
\emph{supplement} of a $t$-$(v,k,\lambda)$ design $\D=(\P,\B)$ is
the design $\overline{\D} = (\P,{\P \choose k}\setminus \B)$ with
parameters $t$-$(v,k,\lambda_{\max}-\lambda)$. Therefore it suffices
to consider existence of $t$-$(v,k,\lambda)$ designs for $\lambda\le
\lambda_{\max}/2$. We shall denote the largest integer $m$ such that
$m\lambda_{\min}\le \lambda_{\max}/2$ by $M$. The \emph{complement}
of $\D$, obtained by taking the complement of each block in~$\P$, is
also a $t$-design and therefore it suffices to consider parameters
with $k\le v/2$.

\section{Designs with parameters $2$-$(55,10,\lambda)$}\label{secresfirst}

Let $t=2$, $v=55$, and $k=10$. Then $\lambda_{\min}=1$,
$\lambda_{\max} = {53\choose 8}=886322710$, and $M=443161355$. A
$2$-$(55,10,1)$ design would have $b=33$ blocks and cannot exist by
Fisher's inequality \cite[Theorem 1.9]{MR07}. Designs
$2$-$(55,10,2)$ are quasi-residuals of symmetric $2$-$(67,12,2)$
designs, which do not exist by the Bruck-Ryser-Chowla theorem
\cite[Theorem II.6.11]{CD07}. By the Hall-Connor theorem
\cite[Theorem II.6.27]{CD07}, quasi-residual designs with
$\lambda=2$ are actually residual, hence $2$-$(55,10,2)$ designs
also do not exist. According to \cite[Table~1.35]{MR07},
$2$-$(55,10,m)$ designs exist for $m=5$, and are unknown for
$m\in\{3,4,6\}$.

\begin{theorem}\label{tm1}
Simple $2$-$(55,10,4)$ designs exist. 
\end{theorem}

\begin{proof}
Let $G_1\cong PSL(2,11)$ be the group of order $660$ generated by
the permutations
\begin{equation}\label{cycle11}
(1,2,\ldots,11)(12,\ldots,22)(23,\ldots,33)(34,\ldots,44)(45,\ldots,55)
\end{equation}
and
\begin{equation*}
\begin{array}{l}
(2,27)(3,37)(4,29)(5,47)(6,53)(7,30)(8,38)(9,32)(11,43)(13,52)\\
(15,48)(17,28)(18,42)(19,49)(20,51)(21,44)(22,31)(23,46)(24,50)\\
(25,54)(26,33)(34,45)(36,39)(41,55).
\end{array}
\end{equation*}
The group $G_1$ has $6$ orbits on $2$-element subsets of
$\P=\{1,\ldots,55\}$. It suffices to consider orbits of $10$-element
subsets whose size does not exceed the number of blocks $b=132$.
This can be accomplished efficiently by an algorithm described
in~\cite{KV16}; there are $97$ such orbits. The $6\times 97$
Kramer-Mesner system has $5$ solutions for $\lambda=4$, giving rise
to three non-isomorphic designs. Two of them have $G_1$ as their
full automorphism group. They are generated by the base blocks $\{1,
2, 5, 8, 10, 12, 23, 42, 45, 50\}$ and $\{1, 2, 3, 9, 18, 27, 28,
33, 38, 43\}$, respectively. The third design has $\Aut(\D)\cong
G_1.\Z_2\cong PGL(2,11)$ and is generated by the two base blocks $\{
1, 2, 4, 5, 16, 18, 24$, $30, 51, 55\}$ and $\{ 1, 2, 6, 8, 24, 27,
38, 40, 50, 53\}$.
\end{proof}

Designs $2$-$(55,10,5)$ can be constructed from the group $G_2\cong
\Z_{55}.\Z_{10}$. There are five non-isomorphic designs with $G_2$
as their full automorphism group.

\begin{theorem}\label{tm2}
Simple $2$-$(55,10,m)$ designs exist for $m\in\{4,5,8,\ldots,M\}$.
\end{theorem}

\begin{proof}
For $m\in\{8,\ldots,58\}$, designs can be constructed from the group
$G_1$. For example, $2$-$(55,10,8)$ designs are obtained as unions
of any two of the designs from Theorem~\ref{tm1}, since they are
disjoint. By the Kramer-Mesner method we found designs for all $m$
in the specified range. Base blocks are available on the web page
referred to in the Introduction.


For $m\ge 59$, we use the group $G_3\cong M_{11}$ of order $7920$
generated by the permutations~\eqref{cycle11} and
\begin{equation*}
\begin{array}{l}
(3,23)(4,20)(5,10)(6,45)(7,34)(8,33)(11,12)(13,26)(14,15)(16,53)\\
(17,43)(18,28)(19,52)(21,49)(22,47)(24,46)(25,32)(27,51)(29,42)\\
(31,41)(36,39)(38,48)(40,54)(44,50).
\end{array}
\end{equation*}
Orbits of size less than $|G_3|$ were computed by the algorithm
from~\cite{KV16}.
There are $367$ orbits of size less than $3960$. Using these orbits
and the Kramer-Mesner method, we found designs for
$m \in\{59,\ldots,479\}$. Base blocks are available on our web page.
There are $13753$ orbits of size $3960$, of which $1647$ form the
blocks of $2$-$(55,10,120)$ designs. These $1647$ designs are
mutually disjoint and disjoint from the previously constructed
designs. Finally, we used a program written in~C to find the
$3686048$ long orbits of size $7920$. Among them, $555578$ orbits
form disjoint $2$-$(55,10,240)$ designs, and $1256273$ pairs of long
orbits can be combined into as many disjoint $2$-$(55,10,480)$
designs. It is clear that by taking unions of the so far constructed
designs, simple $2$-$(55,10,m)$ design with $G_3$ as an automorphism
group can be constructed for any $m\in \{59,\ldots,M\}$.
\end{proof}

Recently, D.~Crnkovi\'{c} and A.~\v{S}vob~\cite{CS17} also found
$2$-$(55,10,m)$ designs for $m\in\{4,10\}$. Thus, the only remaining
open cases are $m\in\{3,6,7\}$.
We tried to construct these designs using various prescribed
automorphism groups, but did not find any examples.

\section{Designs with parameters $3$-$(20,5,\lambda)$}

For $t=3$, $v=20$, and $k=5$, we have $\lambda_{\min}=2$,
$\lambda_{\max}={17\choose 2}=136$, and $M=34$. According to
\cite[Table~4.46]{KL07}, $3$-$(20,5,2m)$ designs exist for
$m\in\{3,\ldots,6,8,\ldots,34\}$. We found designs for two of the
three missing values of $m$.

\begin{theorem}
Simple $3$-$(20,5,2m)$ designs exist for $m\in\{2,7\}$. 
\end{theorem}

\begin{proof}
Let $G$ be the dihedral group of order $38$, generated by the cycle
$(1,2,\ldots,19)$ and the involution
\begin{equation*}
(2,19)(3,18)(4,17)(5,16)(6,15)(7,14)(8,13)(9,12)(10,11).
\end{equation*}
There are $39$ orbits of $3$-element subsets and $444$ orbits of
$5$-element subsets of $\P=\{1,\ldots,20\}$. One solution of the
Kramer-Mesner system for $\lambda=4$ is given by the first $12$ base
blocks in Table~\ref{tab1}. The next $39$ base blocks generate a
$3$-$(20,5,10)$ design, disjoint from the $3$-$(20,5,4)$ design. The
union of these two design is a simple $3$-$(20,5,14)$ design.
\end{proof}

\begin{table}[th]
\begin{center} {\scriptsize \begin{tabular}{lllll} \hline
$\{1, 2, 3, 5, 13\}$ & $\{1, 2, 3, 8, 10\}$ & $\{1, 2, 4, 6, 12\}$ & $\{1, 2, 4, 8, 17 \}$ & $\{1, 2, 4, 10, 20\}$\\
$\{ 1, 2, 5, 7, 15 \}$ & $\{ 1, 2, 5, 9, 16 \}$ & $\{ 1, 2, 6, 8,
20\}$ & $\{ 1, 2, 6, 9, 11 \}$ & $\{ 1, 2, 7, 11, 14 \}$\\
$\{ 1, 3, 6, 9, 15 \}$ & $\{ 1, 4, 9, 13, 20 \}$ & & & \\
\hline
$\{ 1, 2, 3, 5, 12 \}$ & $\{ 1, 2, 3, 5, 16 \}$ & $\{ 1, 2,
3, 6, 16 \}$ & $\{ 1, 2, 3, 6, 17 \}$ & $\{ 1, 2, 3, 7, 12 \}$ \\
$\{ 1, 2, 3, 7, 16 \}$ & $\{ 1, 2, 4, 5, 9 \}$ & $\{ 1, 2, 4, 9, 11
\}$ & $\{ 1, 2, 4, 9, 15 \}$ & $\{ 1, 2, 4, 9, 20 \}$ \\
$\{ 1, 2, 4, 11, 16 \}$ & $\{ 1, 2, 4, 11, 18 \}$ & $\{ 1, 2, 4, 18,
20 \}$ & $\{ 1, 2, 5, 8, 9 \}$ & $\{ 1, 2, 5, 8, 12 \}$ \\
$\{ 1, 2, 5, 8, 15 \}$ & $\{ 1, 2, 5, 8, 16 \}$ & $\{ 1, 2, 6, 14,
15 \}$ & $\{ 1, 2, 7, 12, 13 \}$ & $\{ 1, 2, 7, 12, 14 \}$ \\
$\{ 1, 2, 8, 9, 20 \}$ & $\{ 1, 2, 8, 10, 12 \}$ & $\{ 1, 2, 8, 12,
20 \}$ & $\{ 1, 2, 9, 10, 20 \}$ & $\{ 1, 2, 9, 11, 13 \}$ \\
$\{ 1, 2, 9, 13, 20 \}$ & $\{ 1, 3, 5, 9, 12 \}$ & $\{ 1, 3, 5, 11,
14 \}$ & $\{ 1, 3, 6, 8, 14 \}$ & $\{ 1, 3, 6, 8, 20 \}$ \\
$\{ 1, 3, 6, 13, 16 \}$ & $\{ 1, 3, 6, 14, 16 \}$ & $\{ 1, 3, 6, 16,
20 \}$ & $\{ 1, 3, 6, 17, 20 \}$ & $\{ 1, 3, 7, 16, 20 \}$ \\
$\{ 1, 3, 8, 11, 14 \}$ & $\{ 1, 4, 7, 11, 16 \}$ & $\{ 1, 4, 10,
14, 20 \}$ & $\{ 1, 5, 10, 14, 20 \}$\\
\hline
\end{tabular}}
\vskip 5mm \caption{Base blocks for $3$-$(20,5,\lambda)$ designs,
$\lambda\in\{4,10,14\}$.}\label{tab1}
\end{center}
\end{table}

\begin{corollary}
Simple $3$-$(20,5,2m)$ designs exist for $m\in\{2,\ldots,34\}$.
\end{corollary}

In fact, designs can be constructed from the same group $G$ for all
$m\in\{2,\ldots,34\}$. The only open case is now $m=1$. We did not
find any $3$-$(20,5,2)$ designs by prescribing automorphism groups.
We examined the subgroups of $PGL(2,19)$ operating on $20$ points
and of $AGL(1,19)$ operating on $19$ points, of orders greater than
$19$, systematically.

\section{Designs with parameters $3$-$(21,7,\lambda)$}

For $t=3$, $v=21$, and $k=7$, we have $\lambda_{\min}=15$,
$\lambda_{\max}={18\choose 4}=3060$, and $M=102$. According to
\cite[Table~4.46]{KL07}, $3$-$(21,7,15m)$ designs exist for $46$
values of $m$. Existence is unknown for $56$ values of $m$, starting
with $m\in\{1, 2, 5, 7,\ldots\}$. We found designs for all but the
first of these unknown values.

\begin{theorem}
Simple $3$-$(21,7,15m)$ designs exist for $m\in\{2,\ldots,102\}$.
\end{theorem}

\begin{proof}
Let $G_1\cong A_6$ be the group of order $360$ generated by the
permutations
\begin{equation}\label{cycle5}
(2,3,4,5,6)(7,8,9,10,11)(12,13,14,15,16)(17,18,19,20,21)
\end{equation}
and
\begin{equation*}
(1,5,2)(3,4,6)(9,18,16)(10,13,12)(11,14,19)(15,17,20).
\end{equation*}
There are $12$ orbits of $3$-element subsets and $406$ orbits of
$7$-element subsets of $\P=\{1,\ldots,21\}$. The Kramer-Mesner
system has $56$ solutions for $\lambda=30$, giving rise to $28$
non-isomorphic designs. All of them have $G_1$ as their full
automorphism group. Base blocks for one of the designs are the first
$10$ sets in Table~\ref{tab2}.

The group $G_2\cong S_6$ of order $720$ generated by the
permutations~\eqref{cycle5} and
\begin{equation*}
(1,4)(2,6)(3,5)(10,12)(11,14)(15,20)(16,18)
\end{equation*}
can be used for $m\ge 3$. The Kramer-Mesner system is of size
$11\times 253$. We checked that solutions exist for all
$\lambda=15m$, $m\in\{3,\dots,102\}$. Base blocks for $m=5$ are the
next $17$ sets in Table~\ref{tab2}. Base blocks for the other cases
are available on our web page.
\end{proof}

\begin{table}[th]
\begin{center} {\scriptsize \begin{tabular}{lll}
\hline
$\{1, 2, 3, 4,5,6,7\}$ & $\{1, 2, 3, 7, 8, 13, 17\}$ & $\{1, 2, 3, 7, 9, 11, 14\}$ \\
$\{1, 2, 3, 7, 9, 12, 21 \}$ & $\{1, 2, 7, 8, 10, 16, 21 \}$ & $\{1, 2, 7, 8, 11, 14, 20 \}$ \\
$\{1, 2, 7, 10, 16, 19, 20 \}$ &  $\{1, 7, 8, 12, 13, 18, 19 \}$ & $\{1, 7, 8, 15, 18, 19, 21 \}$ \\
$\{7, 8, 9, 10, 12, 15, 19 \}$ & & \\
\hline
$\{ 1, 2, 3, 4, 7, 8, 20 \}$ &  $\{ 1, 2, 3, 4, 7, 14, 15 \}$ &  $\{ 1, 2, 3, 7, 8, 10, 12 \}$ \\
$\{ 1, 2, 3, 9, 13, 15, 20 \}$ & $\{ 1, 2, 7, 8, 9, 10, 16 \}$ &  $\{ 1, 2, 7, 8, 9, 13, 19 \}$ \\
$\{ 1, 2, 7, 8, 9, 15, 18 \}$ &  $\{ 1, 2, 7, 8, 11, 14, 17 \}$ & $\{ 1, 2, 7, 9, 12, 17, 20 \}$ \\
$\{ 1, 2, 7, 10, 11, 16, 17 \}$ &  $\{ 1, 2, 7, 10, 11, 19, 20 \}$ &  $\{ 1, 2, 7, 10, 16, 19, 20 \}$ \\
$\{ 1, 7, 8, 9, 10, 12, 15 \}$ &  $\{ 1, 7, 8, 9, 17, 19, 21 \}$ &  $\{ 7, 8, 9, 10, 12, 13, 21 \}$ \\
$\{ 7, 8, 9, 10, 17, 18, 20 \}$ & $\{ 7, 8, 9, 13, 17, 19, 21 \}$ & \\
\hline
\end{tabular}}
\vskip 5mm \caption{Base blocks for a $3$-$(21,7,30)$ design with
$G_1$ and a $3$-$(21,7,75)$ design with $G_2$ as automorphism
group.}\label{tab2}
\end{center}
\end{table}

We did not find any $3$-$(21,7,15)$ designs, and the existence
problem is still open.

\section{Designs with parameters
$4$-$(15,5,\lambda)$}\label{sec4-15-5}

For $t=4$, $v=15$, and $k=5$, we have $\lambda_{\min}=1$,
$\lambda_{\max}=11$, and $M=5$. It is known that $4$-$(15,5,m)$
designs do not exist for $m=1$~\cite{MH72} and exist for
$m\in\{3,4,5\}$ \cite[Table~4.46]{KL07}. We settle the remaining
case $m=2$.

\begin{theorem}
Simple $4$-$(15,5,2)$ designs exist. 
\end{theorem}

\begin{proof}
Let $G\cong \Z_3\times S_3$ be the group of order $18$ generated by
the permutations
\begin{equation*}
(4,5,6)(7,8,9)(10,11,12)(13,14,15),
\end{equation*}
\begin{equation*}
(1,4)(2,5)(3,6)(8,13)(9,10)(11,15).
\end{equation*}
The Kramer-Mesner system is of size $84\times 178$ and has $12$
solutions for $\lambda=2$, giving rise to two non-isomorphic designs
with $\Aut(\D)=G$. Base blocks for one of them are given in
Table~\ref{tab3}.
\end{proof}

\begin{table}[ht]
\begin{center}{\scriptsize \begin{tabular}{llll} \hline
$\{ 1, 2, 3, 4, 5 \}$ & $\{ 1, 2, 3, 7, 11 \}$ & $\{ 1, 2, 4, 7, 11 \}$ & $\{ 1, 2, 4, 7, 14 \}$ \\
$\{ 1, 2, 4, 8, 10 \}$ & $\{ 1, 2, 4, 8, 11 \}$ & $\{ 1, 2, 4, 9, 12 \}$ & $\{ 1, 2, 4, 9, 15 \}$ \\
$\{ 1, 2, 4, 10, 13 \}$ & $\{ 1, 2, 4, 12, 14 \}$ & $\{ 1, 2, 4, 13, 15 \}$ & $\{ 1, 2, 7, 8, 9 \}$ \\
$\{ 1, 2, 7, 8, 14 \}$ & $\{ 1, 2, 7, 12, 15 \}$ & $\{ 1, 2, 10, 11, 12 \}$ & $\{ 1, 2, 10, 11, 15 \}$ \\
$\{ 1, 2, 13, 14, 15 \}$ & $\{ 1, 4, 7, 8, 12 \}$ & $\{ 1, 4, 7, 9, 11 \}$ & $\{ 1, 4, 8, 9, 14 \}$ \\
$\{ 1, 4, 8, 12, 13 \}$ & $\{ 1, 4, 8, 13, 14 \}$ & $\{ 1, 4, 9, 10, 11 \}$ & $\{ 1, 4, 11, 12, 15 \}$ \\
$\{ 1, 4, 11, 14, 15 \}$ & $\{ 1, 7, 8, 10, 11 \}$ & $\{ 1, 7, 8, 10, 13 \}$ & $\{ 1, 7, 8, 11, 15 \}$ \\
$\{ 1, 7, 8, 14, 15 \}$ & $\{ 1, 7, 10, 11, 14 \}$ & $\{ 1, 7, 10, 13, 15 \}$ & $\{ 1, 7, 11, 12, 14 \}$ \\
$\{ 1, 7, 12, 13, 14 \}$ & $\{ 1, 10, 11, 13, 15 \}$ & $\{ 7, 8, 9, 10, 13 \}$ & $\{ 7, 8, 11, 12, 13 \}$\\
\hline
\end{tabular}}
\vskip 5mm \caption{Base blocks for a $4$-$(15,5,2)$
design.}\label{tab3}
\end{center}
\end{table}

For $m\in\{3,5\}$ designs can be constructed from the same group
$G$, and for $m=4$ from the group $G_1\cong A_5$ of order $60$.

\section{Designs with parameters $4$-$(16,8,\lambda)$}

For $t=4$, $v=16$, and $k=8$, we have $\lambda_{\min}=15$,
$\lambda_{\max}={12\choose 4}=495$, and $M=16$. By
\cite[Table~4.46]{KL07}, $4$-$(16,8,15m)$ designs exist for
$m\in\{4,\ldots,16\}$. We found designs for $m=3$.

\begin{theorem}
Simple $4$-$(16,8,45)$ designs exist. 
\end{theorem}

\begin{proof}
Let $G\cong \Z_{15}. (\Z_4\times \Z_2)$ be the group of order $120$ generated by the
permutations
\begin{equation*}
(2,3)(4,5,6,7)(8,9,10,11)(12,13,14,15),
\end{equation*}
\begin{equation*}
(1,5)(2,13)(3,11)(6,15)(7,8)(9,14)(10,12).
\end{equation*}
The Kramer-Mesner system is of size $25\times 132$ and has four
solutions for $\lambda=45$. They correspond to four non-isomorphic
designs with $\Aut(\D)=G$. Base blocks for one of them are given in
Table~\ref{tab4}.
\end{proof}

\begin{table}[ht]
\begin{center} {\scriptsize \begin{tabular}{lll} \hline
$\{ 1, 2, 3, 4, 5, 6, 8, 11 \}$ & $\{ 1, 2, 3, 4, 5, 6, 14, 16 \}$ & $\{ 1, 2, 3, 4, 5, 8, 9, 14 \}$ \\
$\{ 1, 2, 3, 4, 5, 8, 12, 13 \}$ & $\{ 1, 2, 3, 4, 5, 10, 12, 14 \}$ & $\{ 1, 2, 3, 4, 5, 11, 15, 16 \}$ \\
$\{ 1, 2, 3, 4, 6, 8, 10, 16 \}$ & $\{ 1, 2, 3, 4, 6, 13, 15, 16 \}$ & $\{ 1, 2, 4, 5, 6, 7, 9, 11 \}$ \\
$\{ 1, 2, 4, 5, 6, 7, 13, 15 \}$ & $\{ 1, 2, 4, 5, 6, 9, 11, 15 \}$ & $\{ 1, 2, 4, 5, 6, 9, 11, 16 \}$ \\
$\{ 1, 2, 4, 5, 6, 10, 13, 16 \}$ & $\{ 1, 2, 4, 5, 7, 8, 9, 16 \}$ & $\{ 1, 2, 4, 5, 8, 12, 15, 16 \}$ \\
$\{ 1, 2, 4, 6, 8, 9, 10, 15 \}$ & & \\
\hline
\end{tabular}}
\vskip 5mm \caption{Base blocks for a $4$-$(16,8,45)$
design.}\label{tab4}
\end{center}
\end{table}

The same group $G$ can be used to construct $4$-$(16,8,15m)$ designs
for $m\in\{4,\ldots,16\}$. We tried many groups for $m\in\{1,2\}$,
but did not find any designs.

\section{Designs with parameters $5$-$(16,7,\lambda)$}

For $t=5$, $v=16$, and $k=7$, we have $\lambda_{\min}=5$,
$\lambda_{\max}={11\choose 2}=55$, and $M=5$. By
\cite[Table~4.46]{KL07}, $5$-$(16,7,5m)$ designs exist for
$m\in\{3,4,5\}$. Here we settle the case $m=2$.

\begin{theorem}
Simple $5$-$(16,7,10)$ designs exist. 
\end{theorem}

\begin{proof}
Let $G\cong (\Z_2\times\Z_2\times\Z_2\times\Z_2).A_4$ be the group
of order $192$ generated by the permutations
\begin{equation*}
(2,3,4)(5,6,7,8,9,10)(11,12,13,14,15,16),
\end{equation*}
\begin{equation*}
(1,5)(2,12)(3,15)(4,8)(6,14)(7,16)(9,10)(11,13).
\end{equation*}
The Kramer-Mesner system is of size $28\times 71$ and has two
solutions for $\lambda=10$. The two designs are isomorphic and have
$\Aut(\D)=G$. Base blocks are listed in Table~\ref{tab5}.
\end{proof}

\begin{table}[ht]
\begin{center} {\scriptsize \begin{tabular}{llll} \hline
$\{ 1, 2, 3, 4, 5, 6, 13 \}$ & $\{ 1, 2, 3, 4, 5, 6, 14 \}$ &  $\{ 1, 2, 3, 5, 6, 7, 11 \}$ & $\{ 1, 2, 3, 5, 6, 8, 9 \}$ \\
$\{ 1, 2, 3, 5, 6, 9, 10 \}$ &  $\{ 1, 2, 3, 5, 6, 9, 12 \}$ & $\{ 1, 2, 3, 5, 6, 10, 15 \}$ & $\{ 1, 2, 3, 5, 6, 14, 16 \}$ \\
$\{ 1, 2, 3, 5, 8, 11, 12 \}$ & $\{ 1, 2, 5, 6, 7, 8, 16 \}$ &  $\{ 1, 2, 5, 6, 7, 9, 14 \}$ &  $\{ 1, 2, 5, 6, 7, 12, 13 \}$ \\
$\{ 1, 2, 5, 6, 7, 14, 15 \}$ & & & \\
\hline
\end{tabular}}
\vskip 5mm \caption{Base blocks for a $5$-$(16,7,10)$
design.}\label{tab5}
\end{center}
\end{table}

The same group $G$ gives designs for $m=5$, and for $m\in\{3,4\}$ a
subgroup of index $2$ can be used. We did not find any designs for
$m=1$.

\section{Designs with parameters $5$-$(17,8,\lambda)$}\label{secreslast}

For $t=5$, $v=17$, and $k=8$, we have $\lambda_{\min}=20$,
$\lambda_{\max}={12\choose 3}=220$, and $M=5$. By
\cite[Table~4.46]{KL07}, $5$-$(17,8,20m)$ designs exist for
$m\in\{3,4,5\}$. Again, we can settle the $m=2$ case.

\begin{theorem}
Simple $5$-$(17,8,40)$ designs exist. 
\end{theorem}

\begin{proof}
Let $G\cong \Z_{17}.\Z_{16}$ be the group of order $272$ generated
by the cycle $(1,2,\ldots,17)$ and the permutation
\begin{equation*}
(2,4,10,11,14,6,16,12,17,15,9,8,5,13,3,7).
\end{equation*}
The Kramer-Mesner system is of size $25\times 95$. It has $61$
solutions for $\lambda=40$, giving rise to $61$ non-isomorphic
designs with $\Aut(\D)=G$. Base blocks for one of the designs are
given in Table~\ref{tab6}.
\end{proof}

\begin{table}[!ht]
\begin{center}
\noindent {\scriptsize \begin{tabular}{llll} \hline
$\{ 1, 2, 3, 4, 5, 6, 7, 10 \}$ &  $\{ 1, 2, 3, 4, 5, 6, 8, 10 \}$ &  $\{ 1, 2, 3, 4, 5, 6, 8, 13 \}$ \\
$\{ 1, 2, 3, 4, 5, 6, 9, 12 \}$ & $\{ 1, 2, 3, 4, 5, 6, 10, 14 \}$ &  $\{ 1, 2, 3, 4, 5, 7, 8, 12 \}$ \\
$\{ 1, 2, 3, 4, 5, 7, 9, 14 \}$ & $\{ 1, 2, 3, 4, 5, 7, 10, 14 \}$ & $\{ 1, 2, 3, 4, 5, 7, 11, 13 \}$ \\
$\{ 1, 2, 3, 4, 5, 7, 14, 15 \}$ & $\{ 1, 2, 3, 4, 5, 9, 10, 12 \}$ &  $\{ 1, 2, 3, 4, 6, 7, 8, 9 \}$ \\
$\{ 1, 2, 3, 4, 6, 7, 9, 16 \}$ &  $\{ 1, 2, 3, 4, 6, 7, 12, 16 \}$ & $\{ 1, 2, 3, 4, 6, 7, 14, 16 \}$ \\
$\{ 1, 2, 3, 4, 6, 8, 14, 16 \}$ & $\{ 1, 2, 3, 4, 6, 9, 11, 15 \}$ & $\{ 1, 2, 3, 4, 7, 8, 9, 12 \}$ \\
\hline
\end{tabular}}
\vskip 5mm \caption{Base blocks for a $5$-$(17,8,40)$
design.}\label{tab6}
\end{center}
\end{table}

The same group can be used for $m\in\{3,4,5\}$. The existence of
$5$-$(17,8,20)$ designs remains open.

\vskip 5mm

\end{document}